\documentclass[11pt,reqno]{amsart}
\usepackage{amsmath,amssymb, amscd, amsmath, color}
\usepackage{mathrsfs, bm}
\usepackage{graphics}
\input epsf

\numberwithin{equation}{section}
\numberwithin{figure}{section}
\newtheorem{theorem}{Theorem}[section]
\newtheorem{lemma}[theorem]{Lemma}
\newtheorem{corollary}[theorem]{Corollary}

\theoremstyle{definition}
\newtheorem{definition}[theorem]{Definition}

\theoremstyle{remark}
\newtheorem{remark}[theorem]{Remark}

\newtheorem{notation}[theorem]{Notation}

\def\Xint#1{\mathchoice
{\XXint\displaystyle\textstyle{#1}}
{\XXint\textstyle\scriptstyle{#1}}
{\XXint\scriptstyle\scriptscriptstyle{#1}}
{\XXint\scriptscriptstyle\scriptscriptstyle{#1}}
\!\int}
\def\XXint#1#2#3{{\setbox0=\hbox{$#1{#2#3}{\int}$ }
\vcenter{\hbox{$#2#3$ }}\kern-.6\wd0}}

\def\dashint{\Xint-}

\begin{document}

\title[$p(x)$-Laplacian PDEs with Discontinuous Coefficients]{Partial Regularity of Solutions to $\bm{p(x)}$-Laplacian PDEs with Discontinuous Coefficients}

\author[C. S. Goodrich]{Christopher S. Goodrich}
%\address{Department of Mathematics\\
%Creighton Preparatory School\\
%Omaha, NE 68114 USA}
\address{School of Mathematics and Statistics\\
University of New South Wales\\
Sydney, NSW 2052 Australia}
\email[Christopher S. Goodrich]{cgoodrich@creightonprep.org; c.goodrich@unsw.edu.au}
%\urladdr{http://www.math.unl.edu/~s-cgoodri4/}
\author[M. A. Ragusa]{M. Alessandra Ragusa}
\address{Dipartimento di Matematica e Informatica\\
Universita di Catania\\
Catania, Italy}
\email[Alessandra Ragusa]{maragusa@dmi.unict.it}
\author[A. Scapellato]{Andrea Scapellato}
\address{Dipartimento di Matematica e Informatica\\
Universita di Catania\\
Catania, Italy}
\email[Andrea Scapellato]{scapellato@dmi.unict.it}
%\date{June 5, 2018}

\begin{abstract}

For $\Omega\subseteq\mathbb{R}^{n}$ an open and bounded region we consider solutions $u\in W_{\text{loc}}^{1,p(x)}\big(\Omega;\mathbb{R}^{N}\big)$, with $N>1$, of the $p(x)$-Laplacian system
\begin{equation}
\nabla\cdot\left(a(x)|Du|^{p(x)-2}Du\right)=0\text{, a.e. }x\in\Omega,\notag
\end{equation}
where concerning the coefficient function $x\mapsto a(x)$ we assume only that
\begin{equation}
a\in W^{1,q}(\Omega)\cap L^{\infty}(\Omega),\notag
\end{equation}
where $q>1$ is essentially arbitrary.  This implies that the coefficient in the PDE can be highly irregular, and yet in spite of this we still recover that
\begin{equation}
u\in\mathscr{C}_{\text{loc}}^{0,\alpha}\big(\Omega_0\big),\notag
\end{equation}
for each $0<\alpha<1$, where $\Omega_0\subseteq\Omega$ is a set of full measure.  Due to the variational methodology that we employ, our results apply to the more general question of the regularity of the integral functional
\begin{equation}
\int_{\Omega}a(x)|Du|^{p(x)}\ dx.\notag
\end{equation}

\end{abstract}

\maketitle

\section{Introduction}

In this paper we consider the partial regularity of weak solutions $u\ : \ \Omega\subseteq\mathbb{R}^{n}\rightarrow\mathbb{R}^{N}$, where $\Omega$ is bounded and open, of the $p(x)$-Laplacian PDE system
\begin{equation}\label{eq1.1}
\nabla\cdot\left(a(x)|Du|^{p(x)-2}Du\right)=0\text{, }x\in\Omega.
\end{equation}
Because our approach is variational we actually recover the regularity of solutions to \eqref{eq1.1} as a corollary to a more general result -- namely, the partial regularity of minimizers of the functional
\begin{equation}\label{eq1.2}
\int_{\Omega}a(x)|Du|^{p(x)}\ dx.
\end{equation}
As with \eqref{eq1.1}, we consider \eqref{eq1.2} in the \emph{vectorial} setting.  The principal novelty of this work is that the coefficient function $x\mapsto a(x)$ appearing in both \eqref{eq1.1} and \eqref{eq1.2} is not assumed to be continuous.  In fact, we assume only that it satisfies
\begin{equation}\label{eq1.3}
a\in W_{\text{loc}}^{1,q}(\Omega)\cap L^{\infty}(\Omega),
\end{equation}
where $q>1$ is essentially arbitrary.  Thus, the coefficient of $|Du|$ can be quite irregular, and yet we are still able to obtain the almost everywhere H\"{o}lder regularity of $u$.  More precisely, for each $0<\alpha<1$ we demonstrate that there exists an open set $\Omega_0\subseteq\Omega$ such that
\begin{equation}
u\in\mathscr{C}_{\text{loc}}^{0,\alpha}\big(\Omega_0\big),\notag
\end{equation}
where $\Omega_0$ is of full measure in the sense that
\begin{equation}
\big|\Omega\setminus\Omega_0\big|=0.\notag
\end{equation}
In fact, as we demonstrate in this paper the precise structure of the singular set $\Omega\setminus\Omega_0$ is intimately connected with the set of irregular points for the gradient, $Da$, of the coeffcient function -- that is, the set
\begin{equation}
\left\{x\in\Omega\ : \ \liminf_{R\to0^+}R^q\dashint_{\mathcal{B}_R}|Da|^q\ dx>\varepsilon_0\right\},\notag
\end{equation}
for some number $\varepsilon_0>0$ to be described in the proof in Section 3.

In a recent paper Goodrich \cite{goodrich5} considered solutions $u\in W^{1,p}(\Omega)$ of the $p$-Laplacian system
\begin{equation}\label{eq1.4}
\nabla\cdot\Big(a(x)|Du|^{p-2}Du\Big)=0\text{, }x\in\Omega,
\end{equation}
and, correspondingly, minimizers of the functional
\begin{equation}
\int_{\Omega}f\big(a(x)|Du|\big)\ dx,\notag
\end{equation}
where $f$ was of class $\mathscr{C}^{2}(\mathbb{R})$, uniformly $p$-convex, and satisfying a standard $p$-growth condition.  (We note that in \cite{goodrich5} there was a misstatement that $f\in\mathscr{C}^{1}(\mathbb{R})$, but, in fact, it should be $f\in\mathscr{C}^{2}(\mathbb{R})$.)  In each case, it was demonstrated that even with the very weak condition \eqref{eq1.3} in force, both a weak solution of the PDE and a minimizer of the function must be of class $\mathscr{C}_{\text{loc}}^{0,\alpha}$, for each $0<\alpha<1$, a.e. on $\Omega$.  At those points where this regularity failed to hold, that is, the singular set, it was shown that this singular set must be a subset of the set
\begin{equation}
\left\{x\in\Omega\ : \ \liminf_{R\to0^+}
R^q\dashint_{\mathcal{B}_R}|Da|^q\ dx>0\right\}.\notag
\end{equation}
Since it is known that
\begin{equation}
\Bigg|\left\{x\in\Omega\ : \ \liminf_{R\to0^+}
R^q\dashint_{\mathcal{B}_R}|Da|^q\ dx>0\right\}\Bigg|=0,\notag
\end{equation}
this establishes that the singular set is Lebesgue null.

Note that in \cite{goodrich5} only the standard growth setting was considered.  But there has been much research in the past 20 years into the so-called nonstandard growth setting, especially the variable growth setting.  In this case, if one considers minimizers of the integral functional
\begin{equation}\label{eq1.5}
\int_{\Omega}f(x,u,Du)\ dx,
\end{equation}
then instead of assuming that, say,
\begin{equation}
\big|f(x,u,\xi)\big|\le C\big(1+|\xi|^2\big)^{\frac{p}{2}},\notag
\end{equation}
which is a standard $p$-growth assumption, one instead assumes that, say,
\begin{equation}
\big|f(x,u,\xi)\big|\le C\big(1+|\xi|^2\big)^{\frac{p(x)}{2}},\notag
\end{equation}
where $p\ : \ \Omega\rightarrow(1,+\infty $ satisfies some sort of continuity condition -- e.g., H\"{o}lder continuity.  These sorts of irregular growth problems were investigated first by Zhikov \cite{zhikov1,zhikov4,zhikov2,zhikov3} and then Coscia and Mingione \cite{cosciamingione1}.  From there many papers have appeared that study the regularity properties of either minimizers of problems such as \eqref{eq1.5} or solutions of problems such as \eqref{eq1.1} -- see, for example, papers by Acerbi and Mingione \cite{acerbimingione2,acerbimingione1}, Eleuteri \cite{eleuteri2}, Eleuteri and Habermann \cite{eleuterihabermann1}, Foss and Fey \cite{feyfoss2}, Goodrich and Scapellato \cite{goodrichscapellato1}, Nio and Usuba \cite{nio1}, Ragusa and Tachikawa \cite{ragusa4,ragusa5,ragusatachikawa1}, Tachikawa \cite{tachikawa1}, Tachikawa and Usuba \cite{tachikawausuba1}, and Usuba \cite{usuba1}.  It should also be noted that while irregular growth problems are interesting to study from a purely mathematical view, they do possess meaningful applications to physical problems such as electrorheological fluids and thermistors.

So, in light of the substantial literature on $p(x)$-type growth problems it seems a natural question to ask whether the results for the $p$-Laplacian system \eqref{eq1.4} can be extended to a $p(x)$-Laplacian system \eqref{eq1.1}.  In fact, there are extra technical difficulties involved in this process due to the fact that the exponent is now variable instead of constant.  Thus, while problems \eqref{eq1.1} and \eqref{eq1.4} appear similar, addressing the regularity of solutions of \eqref{eq1.1} is a more complicated and delicate endeavor.

We conclude by mentioning some of the relevant literature.  For some classical papers addressing the regularity problem one may consult the papers of De Giorgi \cite{degiorgi1}, Giaquinta and Giusti \cite{giaquintagiusti1,giaquintagiusti2}, Giaquinta and Modica \cite{giaquintamodica1,giaquintamodica2,giaquintamodica3,giaquintamodica4}, Schoen and Uhlenbeck \cite{uhlenbeck2}, and Uhlenbeck \cite{uhlenbeck1}.  The monographs by Giaquinta \cite{giaquinta1} and Giusti \cite{giusti1} are useful references for a general overview of the techniques and problems in regularity theory.  The more recent papers by Mingione \cite{mingione1,mingione2}, which also contain substantial background on the regularity problem in a variety of contexts, are also useful references.

More specifically relevant to the results of this paper the study of the general functional \eqref{eq1.5} under weak conditions on the integrand $f$ has seen much work in the past several decades.  For example, the relaxed condition that $f$ be only asymptotically convex has featured prominently in recent years -- see, for example, the papers by Chipot and Evans \cite{chipot1}, Foss \cite{foss1}, Foss, et al. \cite{fossnapoliverde1,fossnapoliverde2,fossgoodrich1,fossgoodrich2}, Goodrich \cite{goodrich3,goodrich2}, Passarelli di Napoli and Verde \cite{passarelliverde1}, Raymond \cite{raymond1}, and Scheven and Schmidt \cite{scheven1,scheven2}.  Furthermore, in addition, to the case of $p(x)$-type growth already mentioned, other nonstandard growth such as either $(p,q)$-growth or $\Phi$-growth has been studied extensively as well -- see, for example, papers by Breit, et al. \cite{breit1}, Diening, et al. \cite{diening3,diening1,diening2}, Fey and Foss \cite{feyfoss1,feyfoss2}, and Giannetti, et al. \cite{giannetti1,giannetti2}.

Since our problems \eqref{eq1.1}--\eqref{eq1.2} also contain discontinuous coefficients in addition to irregular growth, we would like to mention some of the contributions in this direction, too.  One popular and mathematically interesting direction is to consider coefficients of Vanishing Mean Oscillation (VMO) structure -- see, for example, B\"{o}gelein, et al. \cite{bogeleinduzaar1,bogeleinduzaar2}, di Fazio \cite{difazio1}, Goodrich \cite{goodrich1}, Goodrich and Ragusa \cite{goodrichragusa1}, and Ragusa and Tachikawa \cite{ragusa10,ragusa1,ragusa2}.  An alternative direction is to consider coefficients belonging to a combination of either a suitable $L^p(\Omega)$ space or a suitable Sobolev space $W^{1,q}(\Omega)$.  Studies of this type include papers by Cupini, et al. \cite{cupini1}, Eleuteri, et al. \cite{eleuteri1}, Giova \cite{giova2,giova3}, Giova and Passarelli di Napoli \cite{giova1}, Goodrich \cite{goodrich5}, and Passarelli di Napoli \cite{passarelli1,passarelli2,passarelli3,passarelli4}.  In this context, it has generally been the case that some very specific restrictions have been imposed on, say, the degree of integrability of the coefficient function and the degree, if any, of weak differentiability of the coefficient function; sometimes these restrictions have been strongly related to the space dimension $n$.

Of particular relevance to our results here is the recent paper by Eleuteri, Marcellini, and Mascolo \cite{eleuterimarcellini1}.  In that work the authors consider minimizers of functionals having the form
\begin{equation}
\int_{\Omega}a(x)h\big(|Du|\big)^{p(x)}\ dx,\notag
\end{equation}
where $h\ : \ [0,+\infty)\rightarrow[0,+\infty)$ is an increasing, convex function with $h\in W^{2,\infty}$.  In particular, the coefficient function $x\mapsto a(x)$ satisfies $a\in W_{\text{loc}}^{1,r}(\Omega)$ for some $r>n$.  This is a much stronger condition than what we require here, seeing as our equivalent coefficient map can have $r<n$ and, especially, $r\in[1,n)$.

So, in this paper we do \emph{not} require such restrictive assumptions because by using the ideas introduced in \cite{goodrich5} we are able to obtain results under simply assumption \eqref{eq1.3} on the coefficient function $x\mapsto a(x)$.  In fact, we demonstrate that by using a direct argument we are able to obtain regularity results for \eqref{eq1.1}--\eqref{eq1.2} in spite of the highly irregular coefficient function.  Since this is first such demonstration of this fact in the \emph{irregular growth} setting, we hope that the ideas utilized here can be extended to other related problems involving irregular-type growth conditions.

\section{Preliminaries}

We begin by mentioning some preliminary results, which will be useful in our regularity results of Section 3.  The monographs by Giaquinta \cite{giaquinta1} and by Giusti \cite{giusti1} are excellent references for additional information related to the results presented here.  We first mention some notation that we will use.

\begin{notation}\label{notation2.1}
Throughout this paper we will abide by the following conventions.
\begin{itemize}
\item The number $C$ will be a generic constant, whose value may vary from line to line without specific mention.  Without loss of any generality we will always assume that $C\ge1$.  While $C$ may depend on specific structural and growth constants, it will never be allowed to depend on the radius of any ball used as an integration set.

\item Given a point $x_0$ of $\mathbb{R}^{n}$ and a real number $R>0$, we denote by the symbol $\mathcal{B}_{R}\left(x_0\right)$ the open ball defined by
\begin{equation}
\mathcal{B}_{R}\left(x_0\right):=\big\{x\in\mathbb{R}^{n}\ : \ \left|x-x_0\right|<R\big\}.\notag
\end{equation}
Since the center of the ball will be clear from context typically, ordinarily we will write $\mathcal{B}_R$ for $\mathcal{B}_R\left(x_0\right)$.

\item To denote the average value of a function $f\ : \ \Omega\subset\mathbb{R}^n\rightarrow\mathbb{R}$ we write
\begin{equation}
\dashint_{E}f(x)\ dx:=\frac{1}{|E|}\int_{E}f(x)\ dx,\notag
\end{equation}
where $E\subset\Omega$ is assumed to be Lebesgue measurable and $|E|$ denotes the Lebesgue measure of the set $E$.  In addition, given a point $x_0\in\Omega\subseteq\mathbb{R}^{n}$, some real number $R>0$, and a given function $u\ : \ \Omega\rightarrow\mathbb{R}^{N}$, by the symbol $(u)_{x_0,R}\in\mathbb{R}^{N}$ we denote the vector (or real number if $N=1$)
\begin{equation}
(u)_{x_0,R}:=\dashint_{\mathcal{B}_{R}\left(x_0\right)}u(x)\ dx.\notag
\end{equation}
Since the center, $x_0$, of the ball ordinarily will be clear from the context, we will usually suppress the dependence on the center and write $(u)_R$ when we, in fact, mean $(u)_{x_0,R}$.

\item Given matrices $\xi$, $\eta\in\mathbb{R}^{N\times n}$ by $\xi:\eta$ we denote the usual inner product defined by
\begin{equation}
\xi:\eta:=\sum_{i=1}^{N}\sum_{j=1}^{n}\xi_{ij}\eta_{ij}.\notag
\end{equation}
\end{itemize}
\end{notation}

We next recall the notion both of a Morrey space and, relatedly, of a Sobolev-Morrey space.  These will be used extensively in Section 3.

\begin{definition}\label{definition2.2}
For given $p\in[1,+\infty)$ and $0\le\gamma\le n$, the \textbf{\emph{Morrey space}} denoted by the symbol $L^{p,\gamma}\left(E;\mathbb{R}^{N}\right)$ is defined by
\begin{equation}\label{eq2.1}
L^{p,\gamma}\left(E;\mathbb{R}^N\right):=\left\{u\in L^p\left(E;\mathbb{R}^{N}\right)\ : \ \sup_{\substack{y\in E\\ \rho>0}}\frac{1}{\rho^{\gamma}}\int_{E\cap\mathcal{B}_{\rho}(y)}|u|^p\ dx<+\infty\right\},\notag
\end{equation}
where $E\subseteq\mathbb{R}^n$ is a measurable set.  We write $u\in L_{\text{loc}}^{p,\gamma}\left(E;\mathbb{R}^N\right)$ if $u\in L^{p,\gamma}\left(E';\mathbb{R}^N\right)$ for each $E'\Subset E$.
\end{definition}

\begin{definition}\label{definition2.3}
For given $p\in[1,+\infty)$ and $0\le \gamma\le n$, a function $u\in W^{1,p}\left(E,\mathbb{R}^N\right)$, where $E\subseteq\mathbb{R}^{n}$, is said to belong to the \textbf{\emph{Sobolev-Morrey space}} $W^{1,(p,\gamma)}\left(E;\mathbb{R}^N\right)$ provided that $u\in L^{p,
\gamma}\left(E;\mathbb{R}^N\right)$ and $Du\in L^{p,\gamma}\left(E;\mathbb{R}^N\right)$.  We write $u\in W_{\text{loc}}^{1,(p,\gamma)}\left(E;\mathbb{R}^N\right)$ provided that $u\in W^{1,(p,\gamma)}\left(E';\mathbb{R}^N\right)$ for each $E'\Subset E$.
\end{definition}

Next we recall the variable Sobolev spaces.  We refer the reader to the monograph by Diening, et al. \cite{dieningharjulehto1} for additional information on this concept.

\begin{definition}\label{definition2.5aaa}
For a uniformly continuous function $p\ : \ \Omega\rightarrow\mathbb{R}$ the \textbf{\emph{variable Sobolev space}}, denoted by $W_{\text{loc}}^{1,p(x)}(\Omega)$, is defined by
\begin{equation}
W_{\text{loc}}^{1,p(x)}(\Omega):=\left\{u\in L_{\text{loc}}^{1}(\Omega)\ : \ \int_{\Omega}|Du|^{p(x)}\ dx<+\infty\right\}.\notag
\end{equation}
\end{definition}

We next state the hypotheses we impose on problems \eqref{eq1.1}--\eqref{eq1.2}.  We already suggested most of these in Section 1.

\begin{list}{}{\setlength{\leftmargin}{.5in}\setlength{\rightmargin}{0in}}
\item[\textbf{A1:}] The function $p\ : \ \Omega\rightarrow\mathbb{R}$ satisfies each of the following conditions for real constants $\gamma_2>\gamma_1\ge2$.
\begin{enumerate}
\item $\gamma_1\le p(x)\le\gamma_2$ for each $x\in\Omega$

\item $p\in\mathscr{C}^{0,\sigma}(\Omega)$ for some $\sigma\in(0,1)$
\end{enumerate}

\item[\textbf{A2:}] The function $a\ : \ \Omega\rightarrow(0,+\infty)$ satisfies each of the following conditions for real constants $0<a_1<a_2<+\infty$ and $q>1$.
\begin{enumerate}
\item $a\in W_{\text{loc}}^{1,q}(\Omega)\cap L^{\infty}(\Omega)$

\item $a_1\le a(x)\le a_2$, for a.e. $x\in\Omega$
\end{enumerate}
\end{list}

\begin{remark}\label{remark2.5aaa}
Henceforth, we will write, without any loss, $\displaystyle\gamma_1:=\inf_{x\in\Omega}p(x)$ and $\displaystyle\gamma_2:=\sup_{x\in\Omega}p(x)$.
\end{remark}

\begin{remark}\label{remark2.4}
We assume that $p$ is H\"{o}lder continuous, as characterized by condition (A1.2), only for the sake of convenience.  It is entirely possible to assume instead that $p$ satisfies a log-type continuity condition -- see, for example, Fey and Foss \cite{feyfoss2}.
\end{remark}

\begin{remark}\label{remark2.5}
As we already remarked in Section 1, in contrast to the existing work on problems such as \eqref{eq1.1}--\eqref{eq1.2} with discontinuous coefficients (with the exception of \cite{goodrich5}), here we do not request a specific value for the integrability parameter $q$ in condition (A2.1).  Rather, any $q>1$ is allowed.
\end{remark}

We next state two reverse H\"{o}lder inequalities that will be necessary to complete the regularity argument of Section 3.  The first of these, Lemma \ref{lemma2.1}, is a reverse H\"{o}lder inequality for a minimizer, $u$, of the functional \eqref{eq1.2}.  The second of these, Lemma \ref{lemma2.2}, is, on the other hand, a reverse H\"{o}lder inequality for a minimizer, $v$, of the functional
\begin{equation}
\int_{\mathcal{B}_R}|Du|^{p(x)}\ dx\notag
\end{equation}
and satisfying the boundary condition $u\equiv v$ on $\partial\mathcal{B}_R$ for a given ball $\mathcal{B}_R\Subset\Omega$.

\begin{lemma}\label{lemma2.1}
Assume that conditions (A1)--(A2) hold and suppose that $u\in W_{\text{loc}}^{1,p(x)}(\Omega)$ is a minimizer of the functional \eqref{eq1.2}.  Then there exists a number $\delta_0>0$ such that for each $\delta\in\left(0,\delta_0\right)$ it follows both that $Du\in L_{\text{loc}}^{(1+\delta)p(x)}(\Omega)$ and that
\begin{equation}
\dashint_{\mathcal{B}_{R}(y)}|Du|^{(1+\delta)p(x)}\ dx\le C\left(\dashint_{\mathcal{B}_{2R}(y)}|Du|^{p(x)}\ dx\right)^{1+\delta}\notag
\end{equation}
for any $\mathcal{B}_{2R}(y)\Subset\Omega$.
\end{lemma}

\begin{proof}
The proof is similar to \cite[Theorem 3.1]{eleuteri2}, and so, we just provide a brief sketch of the proof.  So, in the usual way we let $s$, $t\in(0,+\infty)$ be selected such that
\begin{equation}
0<\frac{R}{2}\le t<s\le R.\notag
\end{equation}
In addition, let $\eta\in\mathscr{C}_0^{\infty}\left(\mathcal{B}_R\right)$ be a cut-off function with the property that $\eta\equiv0$ for $x\in\Omega\setminus\mathcal{B}_s$, $\eta(x)\equiv1$ for $x\in\mathcal{B}_t$, and $\displaystyle|D\eta|\le\frac{2}{s-t}$ for $x\in\mathcal{B}_s\setminus\mathcal{B}_t$.  Finally define the functions $\varphi$ and $v$, respectively, by
\begin{equation}
\varphi(x):=\eta(x)\big(u(x)-(u)_R\big)\notag
\end{equation}
and
\begin{equation}
v(x):=u(x)-\varphi(x).\notag
\end{equation}
Then we see that
\begin{equation}
\int_{\mathcal{B}_t}|Du|^{p(x)}\ dx\le C\int_{\mathcal{B}_s\setminus\mathcal{B}_t}|Du|^{p(x)}\ dx+\frac{C}{(s-t)^{\gamma_2}}\int_{\mathcal{B}_R}\big|u(x)-(u)_R\big|^{p(x)}\ dx,\notag
\end{equation}
where $\displaystyle \gamma_2:=\sup_{x\in\Omega}p(x)$.  It then follows that
\begin{equation}
\dashint_{\mathcal{B}_{R}(y)}|Du|^{(1+\delta)p(x)}\ dx\le C\left(\dashint_{\mathcal{B}_{2R}(y)}|Du|^{p(x)}\ dx\right)^{1+\delta}\notag
\end{equation}
for any $\mathcal{B}_{2R}(y)\Subset\Omega$, as claimed.
\end{proof}

\begin{lemma}\label{lemma2.2}
Assume that conditions (A1)--(A2) hold and suppose that $w\in W^{1,p(x)}\big(\mathcal{B}_R(x_0)\big)$ for some ball $\mathcal{B}_R(x_0)\Subset\Omega$ and that $w$ is a minimizer of the functional
\begin{equation}
w\mapsto\int_{\mathcal{B}(x_0)}|Dw|^{p(x)}\ dx.\notag
\end{equation}
Suppose, in addition, that $u-w\in W_0^{1,p(x)}\left(\mathcal{B}_R\right)$, where $u\in W^{1,(1+r)p(x)}\left(\mathcal{B}_{2R}\right)$ for some $r>0$.  Then there exists a number $\delta_0>0$ such that for each $\delta\in\left(0,\delta_0\right)$ it holds that
\begin{equation}
\dashint_{\mathcal{B}_R}|Du-Dw|^{(1+\delta)p(x)}\ dx\le C\left(\dashint_{\mathcal{B}_R}|Du-Dw|^{p(x)}\ dx\right)^{1+\delta}+C\dashint_{\mathcal{B}_{2R}}\left(1+|Du|^{p(x)}\right)^{1+\delta}\ dx.\notag
\end{equation}
\end{lemma}

\begin{proof}
Essentially combining the arguments in \cite[Theorem 3.1]{eleuteri2} and \cite[Lemma 2.15]{goodrich4} we are able to obtain this result.  Therefore, we omit the proof of this lemma.
\end{proof}

Finally, the following lemma will be useful in Section 3.  It relates the Sobolev-Morrey spaces to the H\"{o}lder spaces.

\begin{lemma}\label{lemma2.9}
Assume that $\Omega$ is open, bounded, and has a Lipschitz boundary.  In addition, assume that $\beta\in(n-k,n)$ for some $n>k\ge2$.  Then it holds that
\begin{equation}
W^{1,(k,\beta)}\left(\Omega;\mathbb{R}^{N}\right)\subseteq\mathscr{C}^{0,1-\frac{n-\beta}{k}}\left(\overline{\Omega};\mathbb{R}^{N}\right).\notag
\end{equation}
\end{lemma}

\section{Main Results}

We begin this section by stating and proving the primary result of this paper -- namely, the partial H\"{o}lder continuity of a minimizer, $u$, of integral functional \eqref{eq1.2}.  As a consequence of Theorem \ref{theorem3.1} we then obtain as a corollary the partial regularity of any weak solution of the $p(x)$-Laplacian PDE \eqref{eq1.1}.

\begin{theorem}\label{theorem3.1}
Let $u\in W_{\text{loc}}^{1,p(x)}(\Omega)$ be a local minimizer of the functional \eqref{eq1.2} and assume that conditions (A1)--(A2) are true.  Then there is an open set $\Omega_0\subseteq\mathbb{R}^{n}$ such that for each $n-\gamma_1<\alpha<n$ it holds that $\big|\Omega\setminus\Omega_0\big|=0$ and $u\in\mathscr{C}_{\text{loc}}^{0,1-\frac{n-\alpha}{\gamma_1}}\big(\Omega_0\big)$, where
\begin{equation}
\Omega\setminus\Omega_0\subseteq\left\{x\in\Omega\ : \ \liminf_{R\to0^+}R^q\dashint_{\mathcal{B}_R}|Da|^q\ dx>0\right\}.\notag
\end{equation}
\end{theorem}

\begin{proof}
In what follows we will let $w$ be a minimizer for the functional
\begin{equation}\label{eq3.1}
\int_{\mathcal{B}_R}|Dw|^{p(x)}\ dx,
\end{equation}
where $w\equiv u$ on $\partial\mathcal{B}_R$.  As in the statement of the theorem the function $u$ will be a minimizer of functional \eqref{eq1.2} -- i.e., $u$ is a minimizer for
\begin{equation}
\int_{\Omega}a(x)|Du|^{p(x)}\ dx.\notag
\end{equation}

Now, in what follows, we will assume that $\displaystyle x_0\notin\left\{x\in\Omega\ : \ \liminf_{R\to0^+}R^q\dashint_{\mathcal{B}_R}|Da|^q\ dx>0\right\}$ is fixed but otherwise arbitrary -- i.e., we select $x_0$ such that for each $\varepsilon_1>0$ given there exists a number $R_0>0$ such that
\begin{equation}
R_0^q\dashint_{\mathcal{B}_{R_0}\left(x_0\right)}|Da|^q\ dx<\varepsilon_1.\notag
\end{equation}
Moreover, we will let $R:=R\left(x_0,\Omega\right)$ be chosen such that $\displaystyle0<R<\frac{1}{4}\textsf{dist}\left(x_0,\partial\Omega\right)$.  In addition, since $w$ is a minimizer of problem \eqref{eq3.1}, we recall that $w$ satisfies the \emph{a priori} estimate (see Coscia and Mingione \cite{cosciamingione1} or Ragusa, et al. \cite{ragusa5})
\begin{equation}\label{eq3.2}
\int_{\mathcal{B}_{\rho}}|Dw|^{\rho_2(4R)}\ dx\le C\left(\frac{\rho}{R}\right)^{n-\tau}\left[\int_{\mathcal{B}_R}|Dw|^{\rho_2(4R)}\ dx+R^{n-\tau}\right],
\end{equation}
for each $\displaystyle0<\rho<\frac{R}{4}$ and any $\tau>0$.  Henceforth, we let $\delta>0$ be a number chosen such that
\begin{equation}
\delta<\frac{\sigma}{n}.\notag
\end{equation}
In addition, similar to Ragusa, et al. \cite{ragusa5}, we let $\varepsilon>0$ be a number such that $\displaystyle\varepsilon<\frac{\delta}{2}$, and we assume from now on that $R>0$ is chosen sufficiently small such that
\begin{equation}
(1+\varepsilon)\rho_2(4R)\le\left(1+\frac{\delta}{2}\right)\rho_2(4R)<(1+\delta)\rho_1(4R)\le(1+\delta)p(x),\notag
\end{equation}
for all $x\in\mathcal{B}_R\left(x_0\right)$, where we define the numbers $\rho_1(r)$ and $\rho_2(r)$, respectively, by
\begin{equation}
\rho_1(r):=\inf_{x\in\mathcal{B}_r}p(x)\notag
\end{equation}
and
\begin{equation}
\rho_2(r):=\sup_{x\in\mathcal{B}_r}p(x).\notag
\end{equation}

So, to begin the proof we use estimate \eqref{eq3.2} above, together with \cite[Lemma 8.6]{giusti1} (since $\xi\mapsto|\xi|^{\rho_2(4R)}$ is of class $\mathscr{C}^{2}$ since $\rho_2(4R)\ge2$), to calculate
\begin{equation}\label{eq3.3}
\begin{split}
&\int_{\mathcal{B}_{\rho}}|Du|^{\rho_2(4R)}\ dx\\
&\le C\int_{\mathcal{B}_{\rho}}|Dw|^{\rho_2(4R)}\ dx+C\int_{\mathcal{B}_{\rho}}|Du-Dw|^{\rho_2(4R)}\ dx\\
&\le C\int_{\mathcal{B}_{\rho}}|Dw|^{\rho_2(4R)}\ dx+C\int_{\mathcal{B}_R}|Du-Dw|^2\left(|Du|^{\rho_2(4R)-2}+|Dw|^{\rho_2(4R)-2}\right)\ dx\\
&\le C\left(\frac{\rho}{R}\right)^{n-\tau}\left[\int_{\mathcal{B}_R}|Dw|^{\rho_2(4R)}\ dx+R^{n-\tau}\right]\\
&\quad\quad+C\int_{\mathcal{B}_R}|Du-Dw|^2\left(|Du|^{\rho_2(4R)-2}+|Dw|^{\rho_2(4R)-2}\right)\ dx\\
&\le C\left(\frac{\rho}{R}\right)^{n-\tau}\left[\int_{\mathcal{B}_R}|Dw|^{\rho_2(4R)}\ dx+R^{n-\tau}\right]\\
&\quad\quad+C\int_{\mathcal{B}_R}\left\{|Du|^{\rho_2(4R)}-|Dw|^{\rho_2(4R)}-\frac{\partial}{\partial\xi}\left[|Dw|^{\rho_2(4R)}\right]:[Du-Dw]\right\}\ dx\\
&\le C\left(\frac{\rho}{R}\right)^{n-\tau}\left[\int_{\mathcal{B}_R}|Dw|^{\rho_2(4R)}\ dx+R^{n-\tau}\right]+C\int_{\mathcal{B}_R}\left\{|Du|^{\rho_2(4R)}-|Dw|^{\rho_2(4R)}\right\}\ dx\\
&\quad\quad+C\int_{\mathcal{B}_R}\frac{\partial}{\partial\xi}\left[|Dw|^{p(x)}-|Dw|^{\rho_2(4R)}\right]:[Du-Dw]\ dx\\
&\quad\quad\quad\quad-C\underbrace{\int_{\mathcal{B}_R}\frac{\partial}{\partial\xi}\left[|Dw|^{p(x)}\right]:[Du-Dw]\ dx}_{=0}\\
&=C\left(\frac{\rho}{R}\right)^{n-\tau}\left[\int_{\mathcal{B}_R}|Dw|^{\rho_2(4R)}\ dx+R^{n-\tau}\right]+C\int_{\mathcal{B}_R}\Big\{|Du|^{\rho_2(4R)}-|Dw|^{\rho_2(4R)}\Big\}\ dx\\
&\quad\quad+C\int_{\mathcal{B}_R}\frac{\partial}{\partial\xi}\left[|Dw|^{p(x)}-|Dw|^{\rho_2(4R)}\right]:[Du-Dw]\ dx,
\end{split}
\end{equation}
where to get the final equality we use the fact that $w$ satisfies the Euler-Lagrange equation -- that is, we have that
\begin{equation}
C\int_{\mathcal{B}_R}\frac{\partial}{\partial\xi}\left[|Dw|^{p(x)}\right]:[Du-Dw]\ dx=0.\notag
\end{equation}
Now, we work on estimating from above the quantity
\begin{equation}
C\int_{\mathcal{B}_R}\Big\{|Du|^{\rho_2(4R)}-|Dw|^{\rho_2(4R)}\Big\}\ dx,\notag
\end{equation}
which appears on the right-hand side of \eqref{eq3.3}.  To this end we write
\begin{equation}\label{eq3.4}
\begin{split}
C\int_{\mathcal{B}_R}\Big\{|Du|^{\rho_2(4R)}-|Dw|^{\rho_2(4R)}\Big\}\ dx&=\frac{C}{(a)_R}\int_{\mathcal{B}_R}\Big\{(a)_R|Du|^{\rho_2(4R)}-(a)_R|Dw|^{\rho_2(4R)}\Big\}\ dx\\
&=\frac{C}{(a)_R}\int_{\mathcal{B}_R}(a)_R|Du|^{\rho_2(4R)}-a(x)|Du|^{p(x)}\ dx\\
&\quad\quad+\frac{C}{(a)_R}\int_{\mathcal{B}_R}a(x)|Du|^{p(x)}-(a)_R|Dw|^{p(x)}\ dx\\
&\quad\quad\quad\quad+\frac{C}{(a)_R}\int_{\mathcal{B}_R}(a)_R|Dw|^{p(x)}-(a)_R|Dw|^{\rho_2(4R)}\ dx\\
&=:I+II+III,
\end{split}
\end{equation}
using that $0<(a)_R<+\infty$ for all admissible $R>0$.

We now work to estimate quantities $I$--$III$ from \eqref{eq3.4}.  So, first we write
\begin{equation}
\begin{split}
I&=\frac{C}{(a)_R}\int_{\mathcal{B}_R}(a)_R|Du|^{\rho_2(4R)}-a(x)|Du|^{p(x)}\ dx\\
&=\frac{C}{(a)_R}\int_{\mathcal{B}_R}(a)_R|Du|^{\rho_2(4R)}-(a)_R|Du|^{p(x)}\ dx\\
&\quad\quad+\frac{C}{(a)_R}\int_{\mathcal{B}_R}(a)_R|Du|^{p(x)}-a(x)|Du|^{p(x)}\ dx\\
&\le C\int_{\mathcal{B}_R}\big|a(x)-(a)_R\big||Du|^{p(x)}\ dx\\
&\quad\quad+C\int_{\mathcal{B}_R}\big|(a)_R\big|\left||Du|^{\rho_2(4R)}-|Du|^{p(x)}\right|\ dx\\
&=:I'+I'',\notag
\end{split}
\end{equation}
using again, in the last inequality, the boundedness of $(a)_R$.  To estimate $I'$, letting $\kappa>0$ be a sufficiently small constant to be fixed later in the proof, we write
\begin{equation}\label{eq3.5}
\begin{split}
I'&=C\int_{\mathcal{B}_R}\big|a(x)-(a)_R\big||Du|^{p(x)}\ dx\\
&\le C\int_{\mathcal{B}_R}\big|a(x)-(a)_R\big|^{\kappa}|Du|^{p(x)}\\
&\le CR^n\left(\dashint_{\mathcal{B}_R}\big|a(x)-(a)_R\big|^{\frac{\kappa}{\delta}(1+\delta)}\ dx\right)^{\frac{\delta}{1+\delta}}\left(\dashint_{\mathcal{B}_R}|Du|^{(1+\delta)p(x)}\ dx\right)^{\frac{1}{1+\delta}},
\end{split}
\end{equation}
where the initial sequence of calculations is along the lines of \cite[(3.7)]{goodrich5}.  Recalling the definition of the number $\delta$ at the beginning of this proof, we may further refine \eqref{eq3.5} by writing
\begin{equation}\label{eq3.6}
\begin{split}
&CR^n\left(\dashint_{\mathcal{B}_R}\big|a(x)-(a)_R\big|^{\frac{\kappa}{\delta}(1+\delta)}\ dx\right)^{\frac{\delta}{1+\delta}}\left(\dashint_{\mathcal{B}_R}|Du|^{(1+\delta)p(x)}\ dx\right)^{\frac{1}{1+\delta}}\\
&\quad\le CR^n\left(\dashint_{\mathcal{B}_R}\big|a(x)-(a)_R\big|^{\frac{\kappa}{\delta}(1+\delta)}\ dx\right)^{\frac{\delta}{1+\delta}}\left(\left(\dashint_{\mathcal{B}_{2R}}|Du|^{p(x)}\ dx\right)^{1+\delta}\right)^{\frac{1}{1+\delta}}\\
&\quad\le CR^n\left(R^{\frac{\kappa(1+\delta)}{\delta}}\dashint_{\mathcal{B}_R}|Da|^{\frac{\kappa(1+\delta)}{\delta}}\ dx\right)^{\frac{\delta}{1+\delta}}\left(\dashint_{\mathcal{B}_{2R}}|Du|^{p(x)}\ dx\right),
\end{split}
\end{equation}
where to obtain the first inequality we have used the reverse H\"{o}lder inequality for $u$ (i.e., Lemma \ref{lemma2.1}), whereas to obtain the second inequality we have used Poincar\'{e}'s inequality using that $a\in W^{1,q}(\Omega)$, which is applicable since we may assume that $\kappa$ was selected such that
\begin{equation}
\kappa<\min\left\{\frac{\delta q}{1+\delta},1\right\}.\notag
\end{equation}
Then, noting that $Da\in L^q(\Omega)\subseteq L^{\frac{\kappa(1+\delta)}{\delta}}(\Omega)$, it follows that inequality \eqref{eq3.6} can be further rewritten as
\begin{equation}
\begin{split}
&CR^n\left(\dashint_{\mathcal{B}_R}\big|a(x)-(a)_R\big|^{\frac{\kappa}{\delta}(1+\delta)}\ dx\right)^{\frac{\delta}{1+\delta}}\left(\dashint_{\mathcal{B}_R}|Du|^{(1+\delta)p(x)}\ dx\right)^{\frac{1}{1+\delta}}\\
&\quad\le CR^n\left(R^{\frac{\kappa(1+\delta)}{\delta}}\dashint_{\mathcal{B}_R}|Da|^{\frac{\kappa(1+\delta)}{\delta}}\ dx\right)^{\frac{\delta}{1+\delta}}\left(\dashint_{\mathcal{B}_{2R}}|Du|^{p(x)}\ dx\right)\\
&\quad\le CR^{n+\kappa}\left(\dashint_{\mathcal{B}_{R}}|Da|^{\frac{\kappa(1+\delta)}{\delta}}\ dx\right)^{\frac{\delta}{1+\delta}}\left(\dashint_{\mathcal{B}_{2R}}|Du|^{p(x)}\ dx\right),\notag
\end{split}
\end{equation}
which, using that
\begin{equation}
\begin{split}
\left(\dashint_{\mathcal{B}_{R}}|Da|^{\frac{\kappa(1+\delta)}{\delta}}\ dx\right)^{\frac{\delta}{1+\delta}}&\le C\left(\left(\dashint_{\mathcal{B}_R}|Da|^q\ dx\right)^{\frac{\kappa(1+\delta)}{\delta q}}\right)^{\frac{\delta}{1+\delta}}\\
&\le C\left(\dashint_{\mathcal{B}_R}|Da|^q\ dx\right)^{\frac{\kappa}{q}}\\
&\le CR^{-\kappa}\left(R^{q}\dashint_{\mathcal{B}_R}|Da|^q\ dx\right)^{\frac{\kappa}{q}},\notag
\end{split}
\end{equation}
can be rewritten in the form
\begin{equation}\label{eq3.7}
\begin{split}
&CR^n\left(\dashint_{\mathcal{B}_R}\big|a(x)-(a)_R\big|^{\frac{\kappa}{\delta}(1+\delta)}\ dx\right)^{\frac{\delta}{1+\delta}}\left(\dashint_{\mathcal{B}_R}|Du|^{(1+\delta)p(x)}\ dx\right)^{\frac{1}{1+\delta}}\\
&\le CR^{n+\kappa}\left(R^{-\kappa}\left(R^{q}\dashint_{\mathcal{B}_R}|Da|^q\ dx\right)^{\frac{\kappa}{q}}\right)\left(\dashint_{\mathcal{B}_{2R}}|Du|^{p(x)}\ dx\right)\\
&\le C\left(R^q\dashint_{\mathcal{B}_R}|Da|^q\ dx\right)^{\frac{\kappa}{q}}\int_{\mathcal{B}_{2R}}|Du|^{
\rho_2(4R)}\ dx+CR^n\left(R^q\dashint_{\mathcal{B}_R}|Da|^q\ dx\right)^{\frac{\kappa}{q}}.
\end{split}
\end{equation}
Note that in \eqref{eq3.7} we have used the estimate
\begin{equation}
\begin{split}
C\left(R^q\dashint_{\mathcal{B}_R}|Da|^q\ dx\right)^{\frac{\kappa}{q}}\dashint_{\mathcal{B}_{2R}}|Du|^{p(x)}&\le C\left(R^q\dashint_{\mathcal{B}_R}|Da|^q\ dx\right)^{\frac{\kappa}{q}}\dashint_{\mathcal{B}_{2R}}\big(1+|Du|\big)^{p(x)}\\
&\le C\left(R^q\dashint_{\mathcal{B}_R}|Da|^q\ dx\right)^{\frac{\kappa}{q}}\dashint_{\mathcal{B}_{2R}}\big(1+|Du|\big)^{\rho_2(4R)}.\notag
\end{split}
\end{equation}
In addition, we have also used that $\displaystyle\kappa<\frac{\delta q}{1+\delta}$ so that $\displaystyle\frac{\delta q}{(1+\delta)\kappa}>1$.  In summary, then, our estimate from above for quantity $I'$ is
\begin{equation}
\begin{split}
I'=C\int_{\mathcal{B}_R}\big|a(x)-(a)_R\big||Du|^{p(x)}\ dx&\le C\left(R^q\dashint_{\mathcal{B}_R}|Da|^q\ dx\right)^{\frac{\kappa}{q}}\int_{\mathcal{B}_{2R}}|Du|^{\rho_2(4R)}\ dx\\
&\quad\quad\quad\quad\quad\quad\quad\quad\quad\quad\quad\quad+CR^n\left(R^q\dashint_{\mathcal{B}_R}|Da|^q\ dx\right)^{\frac{\kappa}{q}}.\notag
\end{split}
\end{equation}

At the same time, to provide an upper bound for quantity $I''$ we recall that (see \cite[(7)]{cosciamingione1}, for example)
\begin{equation}
\left||Du|^{\rho_2(4R)}-|Du|^{p(x)}\right|\le CR^{\sigma}\left|1+|Du|^{(1+\varepsilon)\rho_2(4R)}\right|,\notag
\end{equation}
and so, write
\begin{equation}
\begin{split}
I''=C\int_{\mathcal{B}_R}\big|(a)_R\big|\left||Du|^{\rho_2(4R)}-|Du|^{p(x)}\right|\ dx&\le CR^{n+\sigma}\dashint_{\mathcal{B}_R}\left(1+|Du|^{(1+\varepsilon)\rho_2(4R)}\right)\ dx\\
&\le CR^{n+\sigma}+CR^{n+\sigma}\dashint_{\mathcal{B}_R}|Du|^{(1+\varepsilon)\rho_2(4R)}\ dx,\notag
\end{split}
\end{equation}
where we have used the fact that $x\mapsto a(x)$ is bounded so that $\big|(a)_R\big|<+\infty$ for all admissible $R>0$.  In addition, using the reverse H\"{o}lder inequality for $u$ together with the definition of the numbers $\delta$ and $\varepsilon$ allows us to write
\begin{equation}
\dashint_{\mathcal{B}_R}|Du|^{(1+\varepsilon)\rho_2(4R)}\ dx\le C+C\dashint_{\mathcal{B}_R}|Du|^{(1+\delta)p(x)}\ dx\le C+C\left(\dashint_{\mathcal{B}_{2R}}|Du|^{p(x)}\ dx\right)^{1+\delta}\notag
\end{equation}
so that
\begin{equation}
\begin{split}
I''=C\int_{\mathcal{B}_R}\big|a(x)\big|\left||Du|^{\rho_2(4R)}-|Du|^{p(x)}\right|\ dx&\le CR^{n+\sigma}+CR^{n+\sigma}\left(\dashint_{\mathcal{B}_{2R}}|Du|^{p(x)}\ dx\right)^{1+\delta}\\
&\le CR^{n+\sigma}+CR^{n+\sigma}\left(1+\dashint_{\mathcal{B}_{2R}}|Du|^{\rho_2(4R)}\ dx\right)^{1+\delta}\\
&\le CR^{n+\sigma}+CR^{\sigma-n\delta}\int_{\mathcal{B}_{2R}}|Du|^{\rho_2(4R)}\ dx,\notag
\end{split}
\end{equation}
using that $Du$ is locally an element of $L^{\rho_2(4R)}(\Omega)$ to deduce that
\begin{equation}
\left(\int_{\mathcal{B}_{2R}}|Du|^{\rho_2(4R)}\ dx\right)^{1+\delta}\le C\int_{\mathcal{B}_{2R}}|Du|^{\rho_2(4R)}\ dx.\notag
\end{equation}
Therefore, from the preceding upper bounds for $I'$ and $I''$ we deduce that
\begin{equation}\label{eq3.8}
\begin{split}
I&\le CR^{n+\sigma}+CR^{\sigma-n\delta}\int_{\mathcal{B}_{2R}}|Du|^{\rho_2(4R)}\ dx+C\left(R^q\dashint_{\mathcal{B}_R}|Da|^q\ dx\right)^{\frac{\kappa}{q}}\int_{\mathcal{B}_{2R}}|Du|^{\rho_2(4R)}\ dx\\
&\quad\quad\quad\quad\quad\quad\quad\quad\quad\quad\quad\quad\quad\quad\quad\quad\quad\quad\quad\quad\quad\quad\quad\quad\quad\quad
+CR^n\left(R^q\dashint_{\mathcal{B}_R}|Da|^q\ dx\right)^{\frac{\kappa}{q}}\\
&\le CR^n+C\left[R^{\sigma-n\delta}+\left(R^q\dashint_{\mathcal{B}_R}|Da|^q\ dx\right)^{\frac{\kappa}{q}}\right]\int_{\mathcal{B}_{2R}}|Du|^{\rho_2(4R)}\ dx+CR^n\left(R^q\dashint_{\mathcal{B}_R}|Da|^q\ dx\right)^{\frac{\kappa}{q}},
\end{split}
\end{equation}
using that $R^{n+\sigma}<R^n$.  And this completes the upper bound for the quantity $I$.

On the other hand, to estimate $II$, by writing as in \eqref{eq3.5} that
\begin{equation}
\big|a(x)-(a)_R\big|\le C\big|a(x)-(a)_R\big|^{\kappa},\notag
\end{equation}
we note that
\begin{equation}\label{eq3.9}
\begin{split}
II&=\frac{C}{(a)_R}\int_{\mathcal{B}_R}a(x)|Du|^{p(x)}-(a)_R|Dw|^{p(x)}\ dx\\
&\le\frac{C}{(a)_R}\int_{\mathcal{B}_R}a(x)|Dw|^{p(x)}-(a)_R|Dw|^{p(x)}\ dx\\
&\le C\int_{\mathcal{B}_R}\big|a(x)-(a)_R\big|^{\kappa}|Dw|^{p(x)}\ dx\\
&\le CR^n\left(\dashint_{\mathcal{B}_R}|Dw|^{p(x)(1+\delta)}\ dx\right)^{\frac{1}{1+\delta}}\left(\dashint_{\mathcal{B}_R}\big|a(x)-(a)_R\big|^{\frac{\kappa(1+\delta)}{\delta}}\ dx\right)^{\frac{\delta}{1+\delta}}\\
&\le CR^n\left(\dashint_{\mathcal{B}_R}|Dw|^{p(x)(1+\delta)}\ dx\right)^{\frac{1}{1+\delta}}\left(\left(\dashint_{\mathcal{B}_R}\big|a(x)-(a)_R\big|^{q}\ dx\right)^{\frac{\kappa(1+\delta)}{q\delta}}\right)^{\frac{\delta}{1+\delta}}\\
&\le CR^n\left(\dashint_{\mathcal{B}_R}|Dw|^{p(x)(1+\delta)}\ dx\right)^{\frac{1}{1+\delta}}\left(R^q\dashint_{\mathcal{B}_R}|Da|^q\ dx\right)^{\frac{\kappa}{q}},
\end{split}
\end{equation}
where we have used both H\"{o}lder's inequality and Poincar\'{e}'s inequality to estimate the second factor appearing above.  We have also used that
\begin{equation}
\int_{\mathcal{B}_R}a(x)|Du|^{p(x)}\ dx\le\int_{\mathcal{B}_R}a(x)|Dw|^{p(x)}\ dx\notag
\end{equation}
by means of the minimality of $u$.  Then the first factor on the right-hand side of inequality \eqref{eq3.9} can be estimated from above by writing
\begin{equation}
\begin{split}
&C\left(\dashint_{\mathcal{B}_R}|Dw|^{p(x)(1+\delta)}\ dx\right)^{\frac{1}{1+\delta}}\\
&\quad\le C\left(\dashint_{\mathcal{B}_R}|Du-Dw|^{p(x)(1+\delta)}\ dx+\dashint_{\mathcal{B}_R}|Du|^{p(x)(1+\delta)}\ dx\right)^{\frac{1}{1+\delta}}\\
&\quad\le C\Bigg(\left(\dashint_{\mathcal{B}_R}|Du-Dw|^{p(x)}\ dx\right)^{1+\delta}+C\dashint_{\mathcal{B}_{2R}}\left(1+|Du|^{p(x)}\right)^{1+\delta}\ dx\\
&\quad\quad\quad\quad\quad\quad\quad\quad\quad\quad\quad\quad\quad\quad\quad\quad\quad\quad\quad\quad\quad\quad\quad\quad\quad+\left(\dashint_{\mathcal{B}_{2R}}|Du|^{p(x)}\ dx\right)^{1+\delta}\Bigg)^{\frac{1}{1+\delta}}\\
&\quad\le C+C\dashint_{\mathcal{B}_R}|Du|^{p(x)}\ dx+C\dashint_{\mathcal{B}_R}|Dw|^{p(x)}\ dx+C\left(\dashint_{\mathcal{B}_{2R}}|Du|^{p(x)(1+\delta)}\ dx\right)^{\frac{1}{1+\delta}}\\
&\quad\quad\quad\quad\quad\quad\quad\quad\quad\quad\quad\quad\quad\quad\quad\quad\quad\quad\quad\quad\quad\quad\quad\quad\quad\quad\quad\quad\quad\quad+C\dashint_{\mathcal{B}_{2R}}|Du|^{p(x)}\ dx\\
&\quad\le C+C\dashint_{\mathcal{B}_R}|Du|^{p(x)}\ dx+C\left(\dashint_{\mathcal{B}_{2R}}|Du|^{p(x)(1+\delta)}\ dx\right)^{\frac{1}{1+\delta}}+\dashint_{\mathcal{B}_{2R}}|Du|^{p(x)}\ dx,\notag
\end{split}
\end{equation}
where we have used the boundary-type reverse H\"{o}lder inequality for $w$ together with the minimality of $w$ in the final inequality to write
\begin{equation}
\int_{\mathcal{B}_R}|Dw|^{p(x)}\ dx\le\int_{\mathcal{B}_R}|Du|^{p(x)}\ dx.\notag
\end{equation}
In addition, we can reapply the reverse H\"{o}lder inequality for $u$ to write
\begin{equation}
\begin{split}
C\left(\dashint_{\mathcal{B}_{2R}}|Du|^{p(x)(1+\delta)}\ dx\right)^{\frac{1}{1+\delta}}&\le C\left(\left(\dashint_{\mathcal{B}_{4R}}|Du|^{p(x)}\right)^{1+\delta}\right)^{\frac{1}{1+\delta}}\\
&\le C\dashint_{\mathcal{B}_{4R}}|Du|^{p(x)}\ dx.\notag
\end{split}
\end{equation}
Therefore, all in all, we deduce that
\begin{equation}\label{eq3.10}
\begin{split}
II&\le CR^n\left(\dashint_{\mathcal{B}_R}|Dw|^{p(x)(1+\delta)}\ dx\right)^{\frac{1}{1+\delta}}\left(R^q\dashint_{\mathcal{B}_R}|Da|^q\ dx\right)^{\frac{\kappa}{q}}\\
&\le CR^n\left(C+C\dashint_{\mathcal{B}_R}|Du|^{p(x)}\ dx+C\left(\dashint_{\mathcal{B}_{2R}}|Du|^{p(x)(1+\delta)}\ dx\right)^{\frac{1}{1+\delta}}\right)\left(R^q\dashint_{\mathcal{B}_R}|Da|^q\ dx\right)^{\frac{\kappa}{q}}\\
&\le CR^n\left(1+\dashint_{\mathcal{B}_{4R}}|Du|^{p(x)}\ dx\right)\left(R^q\dashint_{\mathcal{B}_R}|Da|^q\ dx\right)^{\frac{\kappa}{q}}\\
&\le CR^n\left(R^q\dashint_{\mathcal{B}_R}|Da|^q\ dx\right)^{\frac{\kappa}{q}}+C\left(R^q\dashint_{\mathcal{B}_R}|Da|^q\ dx\right)^{\frac{\kappa}{q}}\int_{\mathcal{B}_{4R}}|Du|^{p(x)}\ dx\\
&\le CR^n\left(R^q\dashint_{\mathcal{B}_R}|Da|^q\ dx\right)^{\frac{\kappa}{q}}+C\left(R^q\dashint_{\mathcal{B}_R}|Da|^q\ dx\right)^{\frac{\kappa}{q}}\int_{\mathcal{B}_{4R}}|Du|^{\rho_2(4R)}\ dx.
\end{split}
\end{equation}
Thus, inequality \eqref{eq3.10} is our upper bound for quantity $II$.

As for quantity $III$ we estimate
\begin{equation}\label{eq3.11}
\begin{split}
III=\frac{C}{(a)_R}\int_{\mathcal{B}_R}(a)_R|Dw|^{p(x)}-(a)_R|Dw|^{\rho_2(4R)}\ dx&\le C\int_{\mathcal{B}_R}\left||Dw|^{p(x)}-|Dw|^{\rho_2(4R)}\right|\ dx\\
&\le CR^{n+\sigma}\dashint_{\mathcal{B}_R}\left(1+|Dw|^{(1+\varepsilon)\rho_2(4R)}\right)\ dx\\
&\le CR^{n+\sigma}+CR^{n+\sigma}\dashint_{\mathcal{B}_R}|Dw|^{(1+\varepsilon)\rho_2(4R)}\ dx\\
&\le CR^{n+\sigma}+CR^{n+\sigma}\dashint_{\mathcal{B}_R}|Dw|^{(1+\delta)p(x)}\ dx,
\end{split}
\end{equation}
where we have again used the relationship between $\delta$ and $\varepsilon$.  Now, to estimate the second term appearing on the right-hand side of inequality \eqref{eq3.11} we write
\begin{equation}
\begin{split}
&CR^{n+\sigma}\dashint_{\mathcal{B}_R}|Dw|^{(1+\delta)p(x)}\ dx\\
&\quad\le CR^{n+\sigma}\left(\left(\dashint_{\mathcal{B}_R}|Du-Dw|^{p(x)}\ dx\right)^{1+\delta}+\dashint_{\mathcal{B}_R}\left(1+|Du|^{p(x)}\right)^{1+\delta}\ dx\right)\\
&\quad\le CR^{n+\sigma}\left(\left(\dashint_{\mathcal{B}_R}|Du|^{p(x)}\ dx\right)^{1+\delta}+\left(\dashint_{\mathcal{B}_{2R}}|Du|^{(1+\delta)p(x)}\ dx\right)\right)+CR^{n+\sigma}\\
&\quad\le CR^{n+\sigma}+CR^{n+\sigma}\left(\left(\dashint_{\mathcal{B}_{4R}}|Du|^{\rho_2(4R)}\ dx\right)^{1+\delta}+\left(\dashint_{\mathcal{B}_{4R}}|Du|^{p(x)}\ dx\right)^{1+\delta}\right)\\
&\quad\le CR^{n+\sigma}+CR^{n+\sigma}\left(\dashint_{\mathcal{B}_{4R}}|Du|^{\rho_2(4R)}\ dx\right)^{1+\delta},\notag
\end{split}
\end{equation}
where we have used the reverse H\"{o}lder inequalities both for $u$ and for $w$.  Therefore, we see that inequality \eqref{eq3.11} can be written in the form
\begin{equation}\label{eq3.12}
III\le CR^{n+\sigma}+CR^{n+\sigma}\left(\dashint_{\mathcal{B}_{4R}}|Du|^{\rho_2(4R)}\ dx\right)^{1+\delta}.
\end{equation}

Therefore, putting estimates \eqref{eq3.8}, \eqref{eq3.10}, and \eqref{eq3.12} into inequality \eqref{eq3.3} we finally arrive at the estimate
\begin{equation}\label{eq3.13}
\begin{split}
&\int_{\mathcal{B}_{\rho}}|Du|^{\rho_2(4R)}\ dx\\
&\le C\left(\frac{\rho}{R}\right)^{n-\tau}\left[\int_{\mathcal{B}_R}|Dw|^{\rho_2(4R)}\ dx+R^{n-\tau}\right]+I+II+III\\
&\quad\quad+C\int_{\mathcal{B}_R}\frac{\partial}{\partial\xi}\left[|Dw|^{p(x)}-|Dw|^{\rho_2(4R)}\right]:[Du-Dw]\ dx\\
&\le C\left(\frac{\rho}{R}\right)^{n-\tau}\left[\int_{\mathcal{B}_R}|Dw|^{\rho_2(4R)}\ dx+R^{n-\tau}\right]+CR^n\\
&\quad\quad+C\left[R^{\sigma-n\delta}+\left(R^q\dashint_{\mathcal{B}_R}|Da|^q\ dx\right)^{\frac{\kappa}{q}}\right]\int_{\mathcal{B}_{4R}}|Du|^{\rho_2(4R)}\ dx\\
&\quad\quad\quad\quad+CR^n\left(R^q\dashint_{\mathcal{B}_R}|Da|^q\ dx\right)^{\frac{\kappa}{q}}\\
&\quad\quad\quad\quad\quad\quad+C\int_{\mathcal{B}_R}\frac{\partial}{\partial\xi}\left[|Dw|^{p(x)}-|Dw|^{\rho_2(4R)}\right]:[Du-Dw]\ dx,
\end{split}
\end{equation}
where in \eqref{eq3.13} we have used the calculation
\begin{equation}
CR^{n+\sigma}\left(\dashint_{\mathcal{B}_{4R}}|Du|^{\rho_2(4R)}\ dx\right)^{1+\delta}\le CR^{\sigma-n\delta}\int_{\mathcal{B}_{4R}}|Du|^{\rho_2(4R)}\ dx.\notag
\end{equation}
It remains to estimate each of the terms
\begin{equation}\label{eq3.14}
C\int_{\mathcal{B}_R}\frac{\partial}{\partial\xi}\left[|Dw|^{\rho_2(4R)}-|Dw|^{p(x)}\right]:[Du-Dw]\ dx
\end{equation}
and
\begin{equation}\label{eq3.15}
\int_{\mathcal{B}_R}|Dw|^{\rho_2(4R)}\ dx,
\end{equation}
each of which appears in inequality \eqref{eq3.13}.

To estimate the quantity \eqref{eq3.14} we begin by writing
\begin{equation}
\begin{split}
&C\int_{\mathcal{B}_R}\frac{\partial}{\partial\xi}\left[|Dw|^{\rho_2(4R)}-|Dw|^{p(x)}\right]:[Du-Dw]\ dx\\
&\quad\le C\int_{\mathcal{B}_R}\left\{\big|\rho_2(4R)-p(x)\big||Dw|^{\rho_2(4R)-1}+p(x)\left||Dw|^{\rho_2(4R)-1}-|Dw|^{p(x)-1}\right|\right\}|Du-Dw|\ dx\\
&\quad\le CR^{\sigma}\int_{\mathcal{B}_R}\big(1+|Dw|\big)^{(1+\varepsilon)\left(\rho_2(4R)-1\right)}|Du-Dw|\ dx\\
&\quad\le\frac{1}{2}\int_{\mathcal{B}_R}|Du-Dw|^{\rho_2(4R)}\ dx+CR^{\sigma}\int_{\mathcal{B}_R}\big(1+|Dw|\big)^{(1+\varepsilon)\rho_2(4R)}\ dx,\notag
\end{split}
\end{equation}
where we have used some of the corresponding estimates in Ragusa, et al. \cite{ragusa5} 
-- in particular, we have used Young's inequality to obtain the final inequality.  Now recall from \eqref{eq3.3} that we there estimated
\begin{equation}
\begin{split}
C\int_{\mathcal{B}_R}|Du-Dw|^{\rho_2(4R)}\ dx&\le C\int_{\mathcal{B}_R}\Big\{|Du|^{\rho_2(4R)}-|Dw|^{\rho_2(4R)}\Big\}\ dx\\
&\quad\quad+C\int_{\mathcal{B}_R}\frac{\partial}{\partial\xi}\left[|Dw|^{\rho_2(4R)}-|Dw|^{p(x)}\right]:[Du-Dw]\ dx.\notag
\end{split}
\end{equation}
So, in particular, taken together the preceding estimates show that the quantity
\begin{equation}
\frac{1}{2}\int_{\mathcal{B}_R}|Du-Dw|^{\rho_2(4R)}\ dx\notag
\end{equation}
can be absorbed into the left-hand side of \eqref{eq3.13} without affecting the subsequent estimates -- e.g., by simply taking the constant $C$ large enough in a suitable manner.  In addition, we may write
\begin{equation}\label{eq3.16}
\begin{split}
&CR^{\sigma}\int_{\mathcal{B}_R}\big(1+|Dw|\big)^{(1+\varepsilon)\rho_2(4R)}\ dx\\
&\quad\le CR^{\sigma}\left[CR^n+C\int_{\mathcal{B}_R}|Dw|^{(1+\varepsilon)\rho_2(4R)}\ dx\right]\\
&\quad\le CR^{\sigma}\left[CR^n+CR^n\dashint_{\mathcal{B}_R}|Dw|^{(1+\delta)p(x)}\ dx\right]\\
&\quad\le CR^n+CR^{n+\sigma}\left[\dashint_{\mathcal{B}_R}|Du-Dw|^{(1+\delta)p(x)}\ dx+\dashint_{\mathcal{B}_R}|Du|^{(1+\delta)p(x)}\ dx\right]\\
&\quad\le CR^n+CR^{n+\sigma}\Bigg[\left(\dashint_{\mathcal{B}_R}|Du-Dw|^{p(x)}\ dx\right)^{1+\delta}+\dashint_{\mathcal{B}_{2R}}\left(1+|Du|^{p(x)}\right)^{1+\delta}\ dx\\
&\quad\quad\quad\quad\quad\quad\quad\quad\quad\quad\quad\quad\quad\quad\quad\quad\quad\quad\quad\quad\quad\quad\quad\quad+\left(\dashint_{\mathcal{B}_{2R}}|Du|^{p(x)}\ dx\right)^{1+\delta}\Bigg]\\
&\quad\le CR^n+CR^{n+\sigma}\Bigg[\left(\dashint_{\mathcal{B}_{2R}}|Du|^{p(x)}\ dx\right)^{1+\delta}+\dashint_{\mathcal{B}_{2R}}|Du|^{(1+\delta)p(x)}\ dx\Bigg]\\
&\quad\le CR^n+CR^{n+\sigma}\left[\left(\dashint_{\mathcal{B}_{4R}}|Du|^{p(x)}\ dx\right)^{1+\delta}\right]\\
&\quad\le CR^n+CR^{\sigma-n\delta}\int_{\mathcal{B}_{4R}}|Du|^{\rho_2(4R)}\ dx,
\end{split}
\end{equation}
where we again use the boundary-type reverse H\"{o}lder inequality both for $w$ and for $u$, and where we note, also once again, that
\begin{equation}
\left(\int_{\mathcal{B}_{4R}}|Du|^{\rho_2(4R)}\ dx\right)^{1+\delta}\le C\int_{\mathcal{B}_{4R}}|Du|^{\rho_2(4R)}\ dx,\notag
\end{equation}
with $C$ independent of $R$, which is due to the fact that (see \cite{cosciamingione1})
\begin{equation}
u\in W^{1,(1+\varepsilon)\rho_2(4R)}\big(\mathcal{B}_{4R}\big)\notag
\end{equation}
for all $\varepsilon>0$ sufficiently small.  We have also used in \eqref{eq3.16} the minimality of $w$.

On the other hand, in order to estimate quantity \eqref{eq3.15} we use an idea from Ragusa, Tachikawa, and Takabayashi \cite{ragusa5}.  In particular, we first recall that $p\in\mathscr{C}^{0,\sigma}(\Omega)$ for some $\sigma\in(0,1)$.  In particular, then, it holds that
\begin{equation}
\big|p(x)-p(y)\big|\le C|x-y|^{\sigma}=:\omega_1\big(|x-y|\big),\notag
\end{equation}
for each $x$, $y\in\Omega$ with $C$ independent of $x$ and $y$.  Notice that $t\mapsto\omega_1(t)$ is an increasing, concave, continuous map satisfying $\omega_1(0)=0$.  This means that for $R>0$ sufficiently small we have that
\begin{equation}
0\le\omega_1(2R)<\delta_0,\notag
\end{equation}
where $\delta_0$ is the upper bound on the amount of higher integrability in the reverse H\"{o}lder inequality for $w$.  Then we see that
\begin{equation}
\begin{split}
\int_{\mathcal{B}_R}|Dw|^{\rho_2(4R)}\ dx&\le CR^n+C\int_{\mathcal{B}_R}|Dw|^{(1+\varepsilon)\rho_2(4R)}\ dx\\
&\le CR^n+CR^n\dashint_{\mathcal{B}_R}|Dw|^{\left(1+\omega_1(2R)\right)p(x)}\ dx\\
&\le CR^n+CR^n\left(\dashint_{\mathcal{B}_{4R}}|Du|^{p(x)}\right)^{1+\omega_1(2R)}\\
&\le CR^n+CR^n\left(\dashint_{\mathcal{B}_{4R}}|Du|^{\rho_2(4R)}\ dx\right)^{1+\omega_1(2R)}\\
&\le CR^n+CR^{-n\omega_1(2R)}\int_{\mathcal{B}_{4R}}|Du|^{\rho_2(4R)}\ dx,\notag
\end{split}
\end{equation}
where we have used the same sequence of estimates as in \eqref{eq3.16} above and, in addition, we have used the calculation
\begin{equation}
\begin{split}
CR^n\left(\dashint_{\mathcal{B}_{4R}}|Du|^{\rho_2(4R)}\ dx\right)^{1+\omega_1(2R)}&\le CR^{n}\left(R^{-n}\int_{\mathcal{B}_{4R}}|Du|^{\rho_2(4R)}\ dx\right)^{1+\omega_1(2R)}\\
&=CR^{-n\omega_1(2R)}\left(\int_{\mathcal{B}_{4R}}|Du|^{\rho_2(4R)}\ dx\right)^{1+\omega_1(2R)}\\
&\le CR^{-n\omega_1(2R)}\int_{\mathcal{B}_{4R}}|Du|^{\rho_2(4R)}\ dx.\notag
\end{split}
\end{equation}
So, in light of this, we are led to the estimate
\begin{equation}\label{eq3.17}
\begin{split}
&\int_{\mathcal{B}_{\rho}}|Du|^{\rho_2(4R)}\ dx\\
&\le C\left(\frac{\rho}{R}\right)^{n-\tau}\left[\left(R^n+R^{-n\delta}\int_{\mathcal{B}_{4R}}|Du|^{\rho_2(4R)}\ dx\right)+R^{n-\tau}\right]+CR^n\\
&\quad\quad+C\left[R^{\sigma-n\delta}+\left(R^q\dashint_{\mathcal{B}_R}|Da|^q\ dx\right)^{\frac{\kappa}{q}}\right]\int_{\mathcal{B}_{4R}}|Du|^{\rho_2(4R)}\ dx\\
&\quad\quad\quad\quad+CR^n\left(R^q\dashint_{\mathcal{B}_R}|Da|^q\ dx\right)^{\frac{\kappa}{q}}\\
&\quad\quad\quad\quad\quad\quad\quad\quad+CR^n+CR^{\sigma-n\delta}\int_{\mathcal{B}_{4R}}|Du|^{\rho_2(4R)}\ dx\\
&\le CR^{n-\tau}+C\left[\left(\frac{\rho}{R}\right)^{n-\tau}R^{-n\delta}+R^{\sigma-n\delta}+\left(R^q\dashint_{\mathcal{B}_R}|Da|^q\ dx\right)^{\frac{\kappa}{q}}\right]\int_{\mathcal{B}_{4R}}|Du|^{\rho_2(4R)}\ dx\\
&\quad\quad\quad\quad+CR^n\left(R^q\dashint_{\mathcal{B}_R}|Da|^q\ dx\right)^{\frac{\kappa}{q}}.
\end{split}
\end{equation}

Now, notice that
\begin{equation}
\omega_1(2R)<c_1R^{\sigma},\notag
\end{equation}
where $c_1$ is some positive constant independent of $R$.  Then
\begin{equation}
R^{-n\omega_1(2R)}<R^{-nc_1R^{\sigma}},\notag
\end{equation}
and so, an application of L'H\^{o}pital's Rule implies that
\begin{equation}
\lim_{R\to0^+}R^{-n\omega_1(2R)}\le\lim_{R\to0^+}R^{-nc_1R^{\sigma}}=1,\notag
\end{equation}
using that $\sigma>0$.  This means that there exists a constant $c_2$, independent of $R$, such that
\begin{equation}
R^{-n\delta}<c_2<+\infty.\notag
\end{equation}
Consequently, in inequality \eqref{eq3.17} we see that
\begin{equation}
\begin{split}
&C\left[\left(\frac{\rho}{R}\right)^{n-\tau}R^{-n\omega_1(2R)}+R^{\sigma-n\delta}+\left(R^q\dashint_{\mathcal{B}_R}|Da|^q\ dx\right)^{\frac{\kappa}{q}}\right]\int_{\mathcal{B}_{4R}}|Du|^{\rho_2(4R)}\ dx\\
&\quad\quad\quad\quad\quad\quad\le C\left[\left(\frac{\rho}{R}\right)^{n-\tau}+R^{\sigma-n\delta}+\left(R^q\d
ashint_{\mathcal{B}_R}|Da|^q\ dx\right)^{\frac{\kappa}{q}}\right]\int_{\mathcal{B}_{4R}}|Du|^{\rho_2(4R)}\ dx.\notag
\end{split}
\end{equation}

Now we wish to inductively iterate estimate \eqref{eq3.17}.  Since this procedure occurs in a relatively standard way we only sketch a few details for completeness.  So, define the function $\varphi\ : \ [0,+\infty)\rightarrow[0,+\infty)$ by
\begin{equation}
\varphi(R):=\int_{\mathcal{B}_R}\big(1+|Du|\big)^{\rho_2(4R)}\ dx.\notag
\end{equation}
Then inequality \eqref{eq3.17} may be rewritten in the form
\begin{equation}\label{eq3.18}
\begin{split}
\varphi(\rho)&\le CR^{n-\tau}+C\left[\left(\frac{\rho}{R}\right)^{n-\tau}+R^{\sigma-n\delta}+\left(R^q\dashint_{\mathcal{B}_R}|Da|^q\ dx\right)^{\frac{\kappa}{q}}\right]\varphi(4R)\\
&\quad\quad\quad\quad\quad\quad\quad\quad\quad\quad\quad\quad\quad\quad\quad\quad\quad\quad\quad\quad\quad\quad\quad+CR^n\left(R^q\dashint_{\mathcal{B}_R}|Da|^q\ dx\right)^{\frac{\kappa}{q}}.
\end{split}
\end{equation}
Since $x_0\notin\Omega\setminus\Omega_0$, it follows that for any number $\varepsilon_1>0$ we may select $R>0$ sufficiently small such that
\begin{equation}\label{eq3.19}
\left(R^q\dashint_{\mathcal{B}_R}|Da|^q\ dx\right)^{\frac{\kappa}{q}}<\varepsilon_1.
\end{equation}
Furthermore, observe that $C$ is henceforth fixed.

Fix the number $\alpha\in\big(n-\gamma_1,n\big)$ and then choose $\displaystyle\eta\in\left(0,\frac{1}{4}\right)$ and $R>0$ such that
\begin{equation}\label{eq3.20} C\left[\eta^{n-\tau}+R^{\sigma-n\delta}+\varepsilon_1\right]<\eta^{\alpha},
\end{equation}
where we use inequality \eqref{eq3.19}; this then fixes the value of $R$, say $R<R_0$.  Then using \eqref{eq3.20} in \eqref{eq3.18} implies that
\begin{equation}\label{eq3.21}
\begin{split}
\varphi(\rho)&\le CR^{n-\tau}+\eta^{\alpha}\varphi(4R)+CR^n\left(R^q\dashint_{\mathcal{B}_R}|Da|^q\ dx\right)^{\frac{\kappa}{q}}\\
&\le CR^{n-\tau}+\eta^{\alpha}\varphi(4R),
\end{split}
\end{equation}
using again that
\begin{equation}
\left(R^q\dashint_{\mathcal{B}_R}|Da|^q\ dx\right)^{\frac{\kappa}{q}}<1\notag
\end{equation}
for some  $R$ sufficiently small since $x_0\notin\Omega\setminus\Omega_0$.  
In addition, we note that
\begin{equation}\label{eq3.22aaa}
\varphi(\rho)=\int_{\mathcal{B}_{\rho}}\big(1+|Du|\big)^{\rho_2(4R)}\ dx\ge\int_{\mathcal{B}_{\rho}}\big(1+|Du|\big)^{\gamma_1}\ dx\ge\int_{\mathcal{B}_{\rho}}|Du|^{\gamma_1}\ dx.
\end{equation}
Then putting estimates \eqref{eq3.21} and \eqref{eq3.22aaa} together we arrive at
\begin{equation}\label{eq3.23aaa}
\int_{\mathcal{B}_{\rho}}|Du|^{\gamma_1}\ dx\le\left(\frac{\rho}{R}\right)^{\alpha}\int_{\mathcal{B}_{4R}}\big(1+|Du|\big)^{\rho_2(4R)}\ dx+CR^{n-\tau}.
\end{equation}
Upon dividing both sides of inequality \eqref{eq3.23aaa} by $\rho^{\alpha}$, executing a standard inductive iteration procedure (see, for example, the conclusion of the proofs of \cite[Theorem 3.1]{fossgoodrich1}, \cite[Theorem 3.1]{goodrich1}, or \cite[Theorem 3.1]{goodrich4}), and then taking the supremum over all $\displaystyle\rho\in\left(0,\frac{1}{4}R_0\right)$, where $R_0$ is the fixed value of $R$, it follows that $Du$ satisfies the Morrey regularity statement $Du\in L_{\text{loc}}^{\gamma_1,\alpha}$.  Consequently, as a standard consequence of this observation we conclude that there exists an open set $\Omega_0\subseteq\Omega$ such that for each $\alpha\in\big(n-\gamma_1,n\big)$ it holds that
\begin{equation}
u\in\mathscr{C}_{\text{loc}}^{0,1-\frac{n-\alpha}{\gamma_1}}\big(\Omega_0\big),\notag
\end{equation}
where $\big|\Omega\setminus\Omega_0\big|=0$ since
\begin{equation}
\Omega\setminus\Omega_0\subseteq\left\{x\in\Omega\ : \ \liminf_{R\to0^+}R^q\dashint_{\mathcal{B}_R}|Da|^q\ dx>0\right\}\notag
\end{equation}
with
\begin{equation}
\left|\left\{x\in\Omega\ : \ \liminf_{R\to0^+}R^q\dashint_{\mathcal{B}_R}|Da|^q\ dx>0\right\}\right|=0.\notag
\end{equation}
And this completes the proof.
\end{proof}

We conclude by applying the previous results to the elliptic system \eqref{eq1.1}, which shows that weak solutions $u\in W_{\text{loc}}^{1,p(x)}(\Omega)$ to the $p(x)$-Laplacian system
\begin{equation}
\nabla\cdot\left(a(x)|Du|^{p(x)-2}|Du|\right)=0\text{, a.e. }x\in\Omega\notag
\end{equation}
are also locally H\"{o}lder continuous under assumptions (A1)--(A2).  Since the proof of this is a consequence of the convexity of the integrand of integral functional \eqref{eq1.2}, we omit the proof of this result.

\begin{corollary}\label{corollary3.2}
Assume that conditions (A1)--(A2) hold.  Let $u\in W_{\text{loc}}^{1,p(x)}(\Omega)$ be a weak solution of the $p(x)$-Laplacian system
\begin{equation}
\nabla\cdot\left(a(x)|Du|^{p(x)-2}|Du|\right)=0\text{, a.e. }x\in\Omega.\notag
\end{equation}
Then there is an open set $\Omega_0\subseteq\mathbb{R}^{n}$ such that for each $0<\alpha<1$ it holds that $\big|\Omega\setminus\Omega_0\big|=0$ and $u\in\mathscr{C}_{\text{loc}}^{0,\alpha}\big(\Omega_0\big)$, where
\begin{equation}
\Omega\setminus\Omega_0\subseteq\left\{x\in\Omega\ : \ \liminf_{R\to0^+}R^q\dashint_{\mathcal{B}_R}|Da|^q\ dx>0\right\}.\notag
\end{equation}
\end{corollary}

\begin{remark}\label{remark3.3}
As mentioned in Section 1, we emphasize that in both Theorem \ref{theorem3.1} for the integral functional \eqref{eq1.2} and Corollary \ref{corollary3.2} for the $p(x)$-Laplacian PDE \eqref{eq1.1} the only restriction on the coefficient $x\mapsto a(x)$ is that it be essentially bounded and once weakly differentiable with integrability exponent $q>1$.  As was detailed in Section 1, this is contrast to the existing literature where more restrictive conditions are imposed on such a coefficient map.
\end{remark}

\textbf{Acknowledgements.} The authors would like to thank the anonymous referee for his or her careful reading of the original manuscript and for several suggestions for its improvement -- and, in particular, for pointing out the paper by Eleuteri, et al. \cite{eleuterimarcellini1}.

\end{document}